\documentclass[11pt]{amsart}
\usepackage{a4wide}

\newcommand{\IN}{\mathbb N}
\newcommand{\C}{\mathcal C}
\newcommand{\w}{\omega}

\newtheorem{theorem}{Theorem}
\newtheorem{proposition}{Proposition}
\newtheorem{lemma}{Lemma}

\newtheorem{question}{Question}
\theoremstyle{definition}
\newtheorem{example}{Example}

\title[A semigroup is finite iff it is chain-finite and antichain-finite]{A semigroup is finite if and only if\\it is chain-finite and antichain-finite}
\author{Iryna Banakh, Taras Banakh and Serhii Bardyla}
\address{I.~Banakh: Pidstryhach Institute for Applied Problems of Mechanics and Mathematics, National Academy of Sciences of Ukraine, Lviv, Naukova 3b, Ukraine}
\email{ibanakh@yahoo.com}  
\address{T.~Banakh: Ivan Franko National University of Lviv (Ukraine) and Jan Kochanowski University in Kielce (Poland)}
\email{t.o.banakh@gmail.com}
\address{S.~Bardyla: University of Vienna, Institute of Mathematics, Kurt G\"{o}del Research Center, Vienna (Austria)}
\email{sbardyla@yahoo.com}
\thanks{The third author was supported by the Austrian Science Fund FWF (Grant   M-2967).}

\begin{document}
\begin{abstract} A subset $A$ of a semigroup $S$ is called a {\em chain} ({\em antichain}) if $xy\in\{x,y\}$ ($xy\notin\{x,y\}$) for any (distinct) elements $x,y\in S$. A semigroup $S$ is called ({\em anti}){\em chain-finite} if $S$ contains no infinite (anti)chains. We prove that each antichain-finite semigroup $S$ is periodic and for every idempotent $e$ of $S$ the set $\sqrt[\infty]{e}=\{x\in S:\exists n\in\IN\;\;(x^n=e)\}$ is finite. This property of antichain-finite semigroups is used to prove that a semigroup is finite if and only if it is chain-finite and antichain-finite. Also we present an example of an antichain-finite semilattice that is not a union of finitely many chains.
\end{abstract}
\maketitle

\section*{Introduction}

In this paper we present a characterization of finite semigroups in terms of finite chains and antichains. These two notions are well-known in the theory of order (see e.g.  \cite[O-1.6]{Bible} or \cite{LMP}) but can also be defined for semigroups. Let us recall that a {\em semigroup} is a set endowed with an associative binary operation $S\times S\to S$, $\langle x,y\rangle\mapsto xy$. A {\em semilattice} is a commutative semigroup whose any element $x$ is an {\em idempotent} (which means that $xx=x$). Each semilattice $S$ carries a natural partial order $\le$ defined by $x\le y$ iff $xy=yx=x$.

A subset $A$ of a semigroup $S$ is defined to be
\begin{itemize}
\item a {\em chain} if $xy\in\{x,y\}$ for any elements $x,y\in A$;
\item an {\em antichain} if $xy\notin\{x,y\}$ for any distinct elements  $x,y\in A$.
\end{itemize}
The definition implies that each chain consists of idempotents.

A semigroup $S$ is defined to be ({\em anti}){\em chain-finite} if it contains no infinite (anti)chains.

Chain-finite semigroups play an important role in studying $\C$-closed and absolutely $\C$-closed semigroups, see \cite{BBm,GutikPagonRepovs2010,Stepp75,Yokoyama2013}. The notion of an antichain-finite semigroup seems to be new.

The principal result of this note is the following theorem characterizing finite semigroups.

\begin{theorem}\label{t:main} A semigroup $S$ is finite if and only if it is chain-finite and antichain-finite.
\end{theorem}

A crucial step in the proof of this theorem is the following proposition describing the (periodic) structure of antichain-finite semigroups.

A semigroup $S$ is called {\em periodic} if for every $x\in S$ there exists $n\in\IN$ such that $x^n$ is an idempotent of $X$. In this case $$S=\bigcup_{e\in E(S)}\!\!\sqrt[\infty]{e},$$
where $E(S)=\{x\in S:xx=x\}$ is the set of idempotents of $S$ and $$\sqrt[\infty]{e}=\{x\in S:\exists n\in\IN\;\;(x^n=e)\}$$
for $e\in E(S)$.

\begin{proposition}\label{p:main} Each antichain-finite semigroup $S$ is periodic and for every $e\in E(S)$ the set $\sqrt[\infty]{e}$ is finite.
\end{proposition}

Theorem~\ref{t:main} and Proposition~\ref{p:main} will be proved in the next section.

Now we present a simple example of an antichain-finite lattice which is not a union of finitely many chains.

\begin{example}\label{ex} Consider the set
$$S=\{\langle 2n-1,0\rangle:n\in\IN\}\cup\{\langle 2n,m\rangle:n,m\in\IN,\;m\le 2n\}$$endowed with the semilattice binary operation
$$
\langle x,i\rangle\cdot\langle y,j\rangle=\begin{cases}
\langle x,i\rangle&\mbox{if $x=y$ and $i=j$};\\
\langle x-1,0\rangle&\mbox{if $x=y$ and $i\ne j$};\\
\langle x,i\rangle&\mbox{if $x<y$};\\
\langle y,j\rangle&\mbox{if $y<x$}.
\end{cases}
$$It is straightforward to check that the semilattice $S$ 
has the following properties:
\begin{enumerate}
\item $S$ is antichain-finite;
\item $S$ has arbitrarily long finite antichains;
\item $S$ is not a union of finitely many chains;
\item the subsemilattice $L=\{\langle 2n-1,0\rangle:n\in\IN\}$ of $S$ is a chain;
\item $S$ admits a homomorphism $r:S\to L$ such that $r^{-1}(\langle x,0\rangle)=\{\langle y,i\rangle\in S:y\in\{x,x+1\}\}$ is finite for every element $\langle x,0\rangle\in L$.
\end{enumerate}
\end{example}

Example~\ref{ex} motivates the following question.

\begin{question} Let $S$ be an antichain-finite semilattice. Is there a finite-to-one  homomorphism $r:S\to Y$ to a semilattice $Y$ which is a finite union of chains?
\end{question}

A function $f:X\to Y$ is called {\em finite-to-one} if for every $y\in Y$ the preimage $f^{-1}(y)$ is finite.

\section*{Proofs of the main results}

In this section we prove some lemmas implying Theorem~\ref{t:main} and Proposition~\ref{p:main}. More precisely, Proposition~\ref{p:main} follows from Lemmas~\ref{l:periodic} and \ref{l:P-fin}; Theorem~\ref{t:main} follows from Lemma~\ref{l:finite}.

\begin{lemma}\label{l:periodic} Every antichain-finite semigroup $S$ is periodic.
\end{lemma}

\begin{proof} Given any element $x\in S$ we should find a natural number $n\in\IN$ such that $x^n$ is an idempotent. First we show that $x^n=x^m$ for some $n\ne m$. Assuming that $x^n\ne x^m$ for any distinct numbers $n,m$, we conclude that the set $A=\{x^n:n\in\IN\}$ is infinite and for any $n,m\in\IN$ we have $x^{n}x^m=x^{n+m}\notin\{x^n,x^m\}$, which means that $A$ is an infinite antichain in $S$. But such an antichain cannot exist as $S$ is antichain-finite. This contradiction shows that $x^n=x^{m}$ for some numbers $n<m$ and then for the number $k=m-n$ we have $x^{n+k}=x^m=x^n$. By induction we can prove that $x^{n+pk}=x^n$ for every $p\in\IN$. Choose any numbers $r,p\in\IN$ such that $r+n=pk$ and observe that $$x^{r+n}x^{r+n}=x^{r+n}x^{pk}=x^rx^{n+pk}=x^rx^n=x^{r+n},$$
which means that $x^{r+n}$ is an idempotent and hence $S$ is periodic.
\end{proof}

An element $1\in S$ is called an {\em identity} of $S$ if $x1=x=1x$ for all $x\in S$.
For a semigroup $S$ let $S^1=S\cup \{1\}$ where $1$ is an element such that $x1=x=1x$ for every $x\in S^1$. If $S$ contains an identity, then we will assume that $1$ is the identity of $S$ and hence $S^1=S$.

For a set $A\subseteq S$ and element $x\in S$ we put $$xA=\{xa:a\in A\}\quad\mbox{and}\quad Ax=\{ax:a\in A\}.$$

For any element $x$ of a semigroup $S$, the set
$$H_x=\{y\in S:yS^1=xS^1\;\wedge\;S^1y=S^1x\}$$is called the {\em $\mathcal H$-class} of $x$. By Lemma I.7.9 \cite{PR}, for every idempotent $e$ its $\mathcal H$-class $H_e$ coincides with the maximal subgroup of $S$ that contains the idempotent $e$.

\begin{lemma}\label{l:He} If a semigroup $S$ is antichain-finite, then for every idempotent $e$ of $S$ its $\mathcal H$-class $H_e$ is finite.
\end{lemma}

\begin{proof} Observe that the set $H_e\setminus\{e\}$ is antichain (this follows from the fact that the left and right shifts in the group $H_e$ are injective). Since $S$ is antichain-finite, the antichain $H_e\setminus\{e\}$ is finite and so is the set $H_e$.
\end{proof}

\begin{lemma}\label{l:ideal} If a semigroup $S$ is antichain-finite, then for every idempotent $e$ in $S$ we have $(H_e\cdot\sqrt[\infty]{e})\cup(\sqrt[\infty]{e}\cdot H_e)\subseteq H_e$.
\end{lemma}

\begin{proof} Given any elements $x\in \sqrt[\infty]{e}$ and $y\in H_e$, we have to show that $xy\in H_e$ and $yx\in H_e$. Since $x\in\sqrt[\infty]{e}$, there exists a number $n\in\IN$ such that $x^n=e$. Then $x^{n+1}S^1=exS^1\subseteq eS^1$ and $eS^1=x^{2n}S^1\subseteq x^{n+1}S^1$, and hence $eS^1=x^{n+1}S$. By analogy we can prove that $S^1e=S^1x^{n+1}$. Therefore, $x^{n+1}\in H_e$.

Then $xy=x(ey)=(xe)y=(xx^n)y=x^{n+1}y\in H_e$ and
$yx=(ye)x=y(ex)=yx^{n+1}\in H_e$.
\end{proof}

\begin{lemma}\label{l:P-fin} If a semigroup $S$ is antichain-finite, then for every idempotent $e\in E(S)$ the set $\sqrt[\infty]{e}$ is finite.
\end{lemma}

\begin{proof} By Lemma~\ref{l:He}, the $\mathcal H$-class $H_e$ is finite. Assuming that $\sqrt[\infty]{e}$ is infinite, we can choose a sequence $(x_n)_{n\in\w}$ of pairwise distinct points of the infinite set $\sqrt[\infty]{e}\setminus H_e$.

Let $P=\{\langle n,m\rangle\in\w\times\w:n<m\}$ and $\chi:P\to 5=\{0,1,2,3,4\}$ be the function defined by
$$\chi(n,m)=\begin{cases}0&\mbox{if $x_nx_m=x_n$};\\
1&\mbox{if $x_mx_n=x_n$};\\
2&\mbox{if $x_nx_m=x_m$};\\
3&\mbox{if $x_mx_n=x_m$};\\
4&\mbox{otherwise}.
\end{cases}
$$By the  Ramsey Theorem 5 \cite[p.16]{Ramsey}, there exists an infinite subset $\Omega\subseteq\w$ such that $\chi[P\cap(\Omega\times\Omega)]=\{c\}$ for some $c\in\{0,1,2,3,4\}$.

If $c=0$, then $x_nx_m=x_n$ for any numbers $n<m$ in $\Omega$. Fix any two numbers $n<m$ in $\Omega$. By induction we can prove that $x_nx_m^p=x_n$ for every $p\in\IN$. Since $x_m\in \sqrt[\infty]{e}$, there exists $p\in\IN$ such that $x_m^p=e$. Then $x_n=x_nx_m^p=x_ne\in H_e$ by Lemma~\ref{l:ideal}. But this contradicts the choice of $x_n$.

By analogy we can derive a contradiction in cases $c\in\{1,2,3\}$.

If $c=4$, then the set $A=\{x_n\}_{n\in\Omega}$ is an infinite antichain in $S$, which is not possible as the semigroup $S$ is antichain-finite.

Therefore, in all five cases we obtain a contradiction, which implies that the set $\sqrt[\infty]{e}$ is finite.
\end{proof}

Our final lemma implies the non-trivial ``if'' part of Theorem~\ref{t:main}.

\begin{lemma}\label{l:finite} A semigroup $S$ is finite if it is chain-finite and antichain-finite.
\end{lemma}

\begin{proof} Assume that $S$ is both chain-finite and antichain-finite. By Lemma~\ref{l:periodic}, the semigroup $S$ is periodic and hence $S=\bigcup_{e\in E(S)}\!\!\sqrt[\infty]{e}$. By Lemma~\ref{l:P-fin}, for every idempotent $e\in E(S)$ the set $\sqrt[\infty]{e}$ is finite. Now it suffices to prove that the set $E(S)$ is finite.

To derive a contradiction, assume that $E(S)$ is infinite and choose a sequence of pairwise distinct idempotents $(e_n)_{n\in\w}$ in $S$. Let $P=\{\langle n,m\rangle\in\w\times \w:n<m\}$ and $\chi:P\to \{0,1,2,3,4,5\}$ be the function defined by the formula
$$\chi(n,m)=\begin{cases}0&\mbox{if $e_ne_m\in\{e_n,e_m\}$ and $e_me_n\in\{e_n,e_m\}$};\\
1&\mbox{if $e_ne_m=e_n$ and $e_me_n\notin\{e_n,e_m\}$};\\
2&\mbox{if $e_ne_m=e_m$ and $e_me_n\notin\{e_n,e_m\}$};\\
3&\mbox{if $e_ne_m\notin\{e_n,e_m\}$ and $e_me_n=e_n$};\\
4&\mbox{if $e_ne_m\notin\{e_n,e_m\}$ and $e_me_n=e_m$};\\
5&\mbox{if $e_ne_m\notin\{e_n,e_m\}$ and $e_me_n\notin\{e_n,e_m\}$}.
\end{cases}
$$

The Ramsey Theorem 5 \cite[p.16]{Ramsey} yields an infinite subset $\Omega\subseteq \w$ such that $\chi[P\cap(\Omega\times\Omega)]=\{c\}$ for some $c\in\{0,1,2,3,4,5\}$.

Depending on the value of $c$, we shall consider six cases.
\smallskip

If $c=0$ (resp. $c=5$), then $\{e_n\}_{n\in\w}$ is an infinite (anti)chain in $S$, which is forbidden by our assumption.
\smallskip

Next, assume that $c=1$. Then $e_ne_m=e_n$ and $e_me_n\notin\{e_n,e_m\}$ for any numbers $n<m$ in $\Omega$. 
For any number $k\in\Omega$, consider the set $Z_k=\{e_ne_k:k<n\in\Omega\}$. Observe that for any $e_ne_k,e_me_k\in Z_k$ we have
$$(e_ne_k)(e_me_k)=e_n(e_ke_m)e_k=e_ne_ke_k=e_ne_k,$$
which means that $Z_k$ is a chain. Since $S$ is chain-finite, the chain
$Z_k$ is finite.

By induction we can construct a sequence of points $(z_k)_{k\in\w}\in\prod_{k\in\w}Z_k$ and a decreasing sequence of infinite sets $(\Omega_k)_{k\in\w}$ such that $\Omega_0\subseteq \Omega$ and for every $k\in\w$ and $n\in\Omega_k$ we have $e_ne_k=z_k$ and $n>k$. Choose an increasing sequence of numbers $(k_i)_{i\in\w}$ such that $k_0\in\Omega_0$ and $k_i\in\Omega_{k_{i-1}}$ for every $i\in\IN$.
We claim that the set $Z=\{z_{k_i}:i\in\w\}$ is a chain. Take any numbers $i,j\in\w$ and choose any number   $n\in\Omega_{k_i}\cap \Omega_{k_j}$.

If $i\le j$, then
$$z_{k_i}z_{k_j}=(e_ne_{k_i})(e_ne_{k_j})=e_n(e_{k_i}e_n)e_{k_j}=e_ne_{k_i}e_{k_j}=e_ne_{k_i}=z_{k_i}.$$

If $i>j$, then $k_i\in \Omega_{k_{i-1}}\subseteq \Omega_{k_j}$ and hence
$$z_{k_i}z_{k_j}=(e_ne_{k_i})(e_ne_{k_j})=e_n(e_{k_i}e_n)e_{k_j}=e_n(e_{k_i}e_{k_j})=e_nz_{k_j}=e_n(e_ne_{k_j})=e_ne_{k_j}=z_{k_j}.$$
In both cases we obtain that $z_{k_i}z_{k_j}\in\{z_{k_i},z_{k_j}\}$, which means that the set $Z=\{z_{k_i}:i\in\w\}$ is a chain. Since $S$ is chain-finite, the set $Z$ is finite.
Consequently, there exists $z\in Z$ such that the set $\Lambda=\{i\in\w:z_{k_i}=z\}$ is infinite. Choose any numbers $i<j$ in the set $\Lambda$ and then choose any number $n\in\Omega_{k_j}\subseteq\Omega_{k_i}$. Observe that $k_j\in \Omega_{k_{j-1}}\subseteq \Omega_{k_i}$ and hence $e_{k_j}e_{k_i}=z_{k_i}=z$. Then
\begin{multline*}
e_{k_j}=e_{k_j}e_{k_j}=(e_{k_j}e_n)e_{k_j}=e_{k_j}(e_ne_{k_j})=e_{k_j}z_{k_j}=e_{k_j}z=\\
=e_{k_j}z_{k_i}=e_{k_j}(e_{k_j}e_{k_i})=e_{k_j}e_{k_i}\notin \{e_{k_j},e_{k_j}\}
\end{multline*} as $c=1$.

By analogy we can prove that the assumption $c\in\{2,3,4\}$ also leads to a contradiction.
\end{proof}


\begin{thebibliography}{12}

\bibitem{BBm} T.~Banakh, S.~Bardyla,
{\em Characterizing chain-finite and chain-compact topological semilattices}, Semigroup Forum {\bf 98} (2019), no. 2, 234--250.

\bibitem{Bible} G.~Gierz, K.H.~Hofmann, K.~Keimel, J.~Lawson, M.~Mislove, D.~Scott, {\em  Continuous lattices and domains}, Encyclopedia of Mathematics and its Applications, 93. Cambridge University Press, Cambridge, 2003.

\bibitem{Ramsey} R.~Graham, B.~Rothschild, J.~Spencer, {\em Ramsey Theory},  John Wiley \& Sons, Inc., Hoboken, NJ, 1990.

\bibitem{GutikPagonRepovs2010}
O.~Gutik, D.~Pagon, D.~Repov\v{s}, {\em On chains in H-closed
topological pospaces}, Order {\bf 27}:1 (2010), 69--81.

\bibitem{LMP} J.~Lawson, M.~Mislove, H.~Priestley, {\em Infinite antichains and duality theories}, Houston J. Math. {\bf 14}:3 (1988), 423--441.

\bibitem{PR} M.~Petrich, N.~Reilly, {\em  Completely regular semigroups}, Canadian Mathematical Society Series of Monographs and Advanced Texts, 23. A Wiley-Interscience Publication. John Wiley \& Sons, Inc., New York, 1999.

\bibitem{Stepp75} J.W.~Stepp, {\em Algebraic maximal semilattices}, Pacific J. Math.
{\bf 58}:1  (1975), 243--248.

\bibitem{Yokoyama2013}
T. Yokoyama, \emph{On the relation between completeness and
$H$-closedness of pospaces without infinite antichains},
Algebra  Discr. Math. {\bf 15}:2 (2013), 287--294
\end{thebibliography}
\end{document}